\documentclass[a4paper, reqno]{amsart}

\usepackage[utf8]{inputenc}
\usepackage[T1]{fontenc} 

\usepackage{etex}
\usepackage{bbm}
\usepackage{pbox}

\newcommand{\bR}{\mathbb{R}}

\newcommand{\bZ}{\mathbb{Z}}

\usepackage{lmodern}
\usepackage{booktabs}
\usepackage[dvipsnames,svgnames,x11names,hyperref]{xcolor}
\usepackage{mathtools,url,graphicx,verbatim,amssymb,enumerate,stmaryrd,nicefrac}
\usepackage[pagebackref,colorlinks,citecolor=Mahogany,linkcolor=Mahogany,urlcolor=Mahogany,filecolor=Mahogany]{hyperref}
\usepackage{microtype}
\usepackage[margin=1.5in]{geometry}

\usepackage{tikz, tikz-cd}
\usetikzlibrary{matrix,calc,arrows}

\newtheorem{theorem}{Theorem}[section]
\newtheorem{lemma}[theorem]{Lemma}
\newtheorem{proposition}[theorem]{Proposition}
\newtheorem{corollary}[theorem]{Corollary}
\newtheorem*{corollary*}{Corollary}

\newtheorem{atheorem}{Theorem}

\theoremstyle{definition}

\theoremstyle{remark}
\newtheorem{example}[theorem]{Example}
\newtheorem{remark}[theorem]{Remark}

\newcommand{\mr}[1]{{\rm #1}}

\newcommand{\lra}{\longrightarrow}

\DeclareMathOperator*{\colim}{colim}

\newcommand{\fr}{\mr{fr}}
\newcommand{\hcoker}{/\!\!/}

\title{Framings of $W_{g,1}$}

 \author{Alexander Kupers}
\email{a.kupers@utoronto.ca}
\address{Department of Computer and Mathematical Sciences \\
	University of Toronto Scarborough \\
	1265 Military Trail \\
	Toronto, ON M1C 1A4 \\ Canada}

 \author{Oscar Randal-Williams}
 \email{o.randal-williams@dpmms.cam.ac.uk}
 \address{Centre for Mathematical Sciences\\
 Wilberforce Road\\
 Cambridge CB3 0WB\\
 UK}

\date{\today}

\begin{document}

\begin{abstract}We compute the set of framings of $W_{g,1} = D^{2n} \# (S^n \times S^n)^{\# g}$, up to homotopy and diffeomorphism relative to the boundary.\end{abstract}
		\subjclass[2010]{57R15, 11F75, 55R40, 57S05}

\maketitle

\tableofcontents

\vspace{-1cm}

\section{Introduction} 

Closed connected orientable surfaces do not admit a framing unless they have genus 1, but compact connected orientable surfaces of any genus with \emph{non-empty} boundary do admit framings and there has been recent interest in understanding the set of such framings up to homotopy and diffeomorphism and, relatedly, the stabilisers of framings with respect to the action of the mapping class group \cite{RWFramed, KawazumiFramings, CalderonSalter, CalderonSalter2, CuadradoSalter}.

The analogues  in higher dimensions of genus $g$ surfaces with one boundary component are the $2n$-manifolds 
\[W_{g,1} \coloneqq D^{2n} \# (S^n \times S^n)^{\# g},\]
which play a distinguished role in the study of diffeomorphism groups of $2n$-manifolds via homological stability \cite{grwstab1, grwstab2}. These manifolds also admit framings, and all framings of $W_{g,1}$ induce the same homotopy class of framing of $TW_{g,1}\vert_{\partial W_{g,1}}$, see Lemma \ref{lem:SpaceThetaStr}. Fixing once and for all a framing $\ell_\partial$ of $TW_{g,1}\vert_{\partial W_{g,1}}$ in this homotopy class, in our work on Torelli groups and diffeomorphism groups of discs \cite{KR-WAlg, KR-WDisks} we have needed to study the moduli spaces of framed manifolds diffeomorphic to $W_{g,1}$ and with boundary condition $\ell_\partial$. The set of path components of this space is the set of orbits for the action of the mapping class group $\pi_0(\mr{Diff}_\partial(W_{g,1}))$ on the set $\mathrm{Str}_\partial^{\fr}(W_{g,1})$ of homotopy classes of framings of $W_{g,1}$ extending $\ell_\partial$. In that work we could get away with qualitative information about this set of path components: that it is finite. In this note we precisely determine it.

\begin{atheorem}\label{thm:main} Let $g \geq 2$ and $n \geq 1$. The action of the mapping class group $\pi_0(\mr{Diff}_\partial(W_{g,1}))$ on the set $\mathrm{Str}_\partial^{\fr}(W_{g,1})$ of homotopy classes of framings extending $\ell_\partial$ has
	\begin{enumerate}[(i)]
		\item two orbits if $n=1, 3,7$ or $n \equiv 0 \pmod 4$; 
		\item one orbit if $n \neq 1,3,7$ and $n \not\equiv 0 \pmod 4$.
	\end{enumerate}
	If $n > 1$ then in fact these hold for $g \geq 1$. 
\end{atheorem}

While Theorem \ref{thm:main} is quite simple to formulate, our main interest is not so much in this statement but rather in related results concerning the stabilisers of framings with respect to the action of the mapping class group, especially in dimensions $2n \geq 6$. These results require substantial background to formulate, and we leave them to the body of the text: the main results in this direction are Propositions \ref{prop:framings-rel-point} and \ref{prop:IdentifyingDoubleBracketEll}, and Corollary \ref{cor:IdentifyingSingleBracketEll}. In Section \ref{sec:ThetaStr} we explain how highly-connected tangential structures can be analysed similarly.

\begin{remark}
The exceptions in Theorem \ref{thm:main} are covered by the following:
\begin{enumerate}[(i)]
\item The cases $g=0$ are as in Table \ref{tab.pi2no}. (This uses that framings of $D^d$ up to homotopy and diffeomorphism, relative to a fixed boundary condition, are in bijection with $\pi_d(\mr{O}(d))$, and that the diffeomorphisms of the disc act trivially on them by Lemma \ref{lem:ThetaActsTriv}.)

\item The case $n=1$ and $g=1$ is given in \cite[Theorem 3.12]{KawazumiFramings}; it is rather complicated.
\end{enumerate}
The case $n=1$ and $g \geq 2$ is contained in \cite[Theorem 2.9]{RWFramed}. We shall nonetheless continue discussing this case, to draw parallels between it and the high-dimensional case which is our main focus.
\end{remark}

\subsection*{Acknowledgements}
This work was supported by the European Research Council under the European Union's Horizon 2020 research and innovation programme [grant number No.\ 756444 to ORW]; by a Philip Leverhulme Prize from the Leverhulme Trust [to ORW]; and by the National Science Foundation [grant number DMS-1803766 to AK].

\section{Recollections and generalities}\label{sec:General}
In this section we recall some notation and results from \cite{KR-WAlg}, in particular Section 8 of that paper, specialized to framings. 

	\subsection{The mapping class group}
	
	The \emph{mapping class group} of $W_{g,1}$ is defined as
	\[\Gamma_g \coloneqq \pi_0(\mr{Diff}_\partial(W_{g,1})) = \pi_1(B\mr{Diff}_\partial(W_{g,1})),\]
where $\mr{Diff}_\partial(W_{g,1})$ is the topological group of diffeomorphisms fixing a neighborhood of the boundary pointwise, in the $C^\infty$-topology. Let us write
\[H_n \coloneqq H_n(W_{g,1};\bZ),\]
which has an action of the mapping class group. In high dimensions we will use an analysis of $\Gamma_g$ due to Kreck \cite{kreckisotopy} (see that paper for details about the definitions of the homomorphisms between the terms; to apply it in this case see \cite[Lemma 1.5]{grwabelian}):

\begin{theorem}[Kreck] \label{thm:kreck} For $2n \geq 6$, the mapping class group $\Gamma_g \coloneqq \pi_0(\mr{Diff}_\partial(W_{g,1}))$ is described by the pair of extensions
	\[1 \lra I_{g} \lra \Gamma_g \lra G'_g \lra 1,\]
	\[1 \lra \Theta_{2n+1} \lra I_{g} \lra \mr{Hom}(H_n,S\pi_n(\mr{SO}(n))) \lra 1,\]
	with $S\pi_n(\mr{SO}(n)) \coloneqq \mr{im}(\pi_n(\mr{SO}(n)) \to \pi_n(\mr{SO}(n+1)))$, described in Table \ref{tab.spino}, and
	\[G'_g \coloneqq \begin{cases} \mr{Sp}_{2g}(\bZ) & \text{if $n$ is 3 or 7,} \\
	\mr{Sp}_{2g}^q(\bZ) & \text{if $n$ is odd but not 3 or 7,} \\
	\mr{O}_{g,g}(\bZ) & \text{if $n$ is even,}\end{cases}\]
	where $\mr{Sp}_{2g}^q(\bZ) \leq \mr{Sp}_{2g}(\bZ)$ denotes the proper subgroup of symplectic matrices which preserve the standard quadratic refinement $\mu_0$.
\end{theorem}

\begin{table}[h]
	\centering
	\caption{The abelian groups $S\pi_n(\mr{SO}(n))$ for $n \geq 1$, with the exceptions that $S\pi_1(\mr{SO}(1)) = 0$, $S\pi_2(\mr{SO}(2)) = 0$ and $S\pi_6(\mr{SO}(6)) = 0$.}
	\label{tab.spino}
	\begin{tabular}{lcccccccc}
		\toprule
		$n \pmod 8$ & 0          & 1          & 2 & 3     & 4 & 5 & 6 & 7     \\ \midrule
		$S\pi_n(\mr{SO}(n))$  & $(\bZ/2)^2$ & $\bZ/2$ & $\bZ/2$ & $\bZ$ & $\bZ/2$ & 0 & $\bZ/2$ & $\bZ$ \\ \bottomrule
	\end{tabular}
\end{table}

In particular, the homomorphism $\Gamma_g \to G'_g \subset \mr{GL}_{2g}(\bZ)$ is given by sending a diffeomorphism to the induced automorphism of $H_n$. Part of Kreck's theorem is the analysis of precisely which automorphisms of $H_n(W_{g,1};\bZ)$ arise from diffeomorphisms: they must preserve the intersection form, and for $n \neq 3,7$ odd they must also preserve a further quadratic refinement. We will discuss this in more detail in Section \ref{sec:AQuadForm}.

For $n=1,2$ acting on middle homology, and preserving the intersection form, gives homomorphisms
\[\Gamma_g \lra G_g' \coloneqq \begin{cases} \mr{Sp}_{2g}(\bZ) & \text{if $n$ is 1,} \\
	\mr{O}_{g,g}(\bZ) & \text{if $n$ is 2,}\end{cases}\]
	and these are still surjective: the first is folklore, the second is \cite[Theorem 2]{Wall4}.

The subgroup $\Theta_{2n+1} \leq \Gamma_g$ corresponds to $\pi_0(\mr{Diff}_\partial(D^{2n}))$. We will make use of the following well-known fact about diffeomorphisms of the disc.

\begin{lemma}\label{lem:ThetaActsTriv}
For $2n \geq 6$ the derivative map 
\[\Theta_{2n+1} = \pi_0(\mr{Diff}_\partial(D^{2n})) \lra \pi_0(\mr{Bun}_\partial(TD^{2n})) = \pi_{2n}(\mr{SO}(2n)).\]
is trivial.
\end{lemma}
\begin{proof}
Apply \cite[Lemma 8.15]{KR-WAlg} with $g=0$, which deduces this from \cite{burglashof}.
\end{proof}

\subsection{Framings}\label{sec:framings} Tangential structures on $2n$-dimensional manifolds can be equivalently described as $\mr{GL}_{2n}(\bR)$-spaces, or as fibrations over $B\mr{O}(2n)$ \cite[Section 4.5]{grwsurvey}. In particular, framings can be described by the $\mr{GL}_{2n}(\bR)$-space given by $\Theta_\mr{fr} = \mr{GL}_{2n}(\bR)$ with action given by right multiplication, or by the fibration $\theta^\mr{fr} \colon E\mr{O}(2n) \to B\mr{O}(2n)$. Though we used the description in terms of the fibration $\theta^\mr{fr}$ in \cite{KR-WAlg,KR-WDisks}, here we find the former description more convenient; in \eqref{eq:FrTRivialised} we will get a homeomorphism rather than a homotopy equivalence. 

A \emph{framing} on $W_{g,1}$ is a map of $\mr{GL}_{2n}(\bR)$-spaces $\ell \colon \mr{Fr}(TW_{g,1}) \to \Theta_\mr{fr}$, with $\mr{Fr}(TW_{g,1})$ the frame bundle of the tangent bundle of $W_{g,1}$, which is a principal $\mr{GL}_{2n}(\bR)$-bundle. We shall fix a boundary condition $\ell_\partial \colon \mr{Fr}(TW_{g,1}|_{\partial W_{g,1}}) \to \Theta_\mr{fr}$ and only consider those $\theta$-structures which extend this boundary condition: the \emph{space of framings of $W_{g,1}$ extending $\ell_\partial$} is defined to be the space 
\[\mr{Bun}_\partial(\mr{Fr}(TW_{g,1}),\Theta_\mr{fr};\ell_\partial)\]
of $\mr{GL}_{2n}(\bR)$-equivariant maps $\mr{Fr}(TW_{g,1}) \to \Theta_\mr{fr}$ extending $\ell_\partial$. We write
\[\mr{Str}^\mr{fr}_\partial(W_{g,1}) \coloneqq \pi_0(\mr{Bun}_\partial(\mr{Fr}(TW_{g,1}),\Theta_\mr{fr}; \ell_\partial))\]
for its set of path components.

The manifold $W_{g,1}$ indeed admits a framing: viewing it as the boundary connect-sum of $g$ copies of the plumbing of $S^n \times D^n$ with itself, it is enough to note that $S^n \times D^n$ may be framed. The following is a special case of \cite[Lemma 8.5 (i)]{KR-WAlg}:

\begin{lemma}\label{lem:SpaceThetaStr}
Up to homotopy there is a unique orientation preserving boundary condition $\ell_\partial$ which extends to a framing $\ell$ on all of $W_{g,1}$.
\end{lemma}

For such a boundary condition $\ell_\partial$, a choice of reference framing $\tau$ provides an isomorphism $\mr{Fr}(TW_{g,1}) \cong W_{g,1} \times \mr{GL}_{2n}(\bR)$ of $\mr{GL}_{2n}(\bR)$-spaces and hence induces a homeomorphism
		\begin{equation}\label{eq:FrTRivialised}
		\mr{Bun}_\partial(\mr{Fr}(TW_{g,1}),\Theta_\mr{fr}; \ell_\partial) \overset{\cong}\lra \mr{map}_\partial(W_{g,1}, \mr{GL}_{2n}(\bR)),
		\end{equation}
		and therefore a bijection 
\[\mr{Str}^\mr{fr}_\partial(W_{g,1}) \overset{\sim}\lra \pi_0(\mr{map}_\partial(W_{g,1}, \mr{SO}(2n))),\]
though one must use it carefully as it depends on the choice of framing $\tau$.

The group $\mr{Diff}_\partial(W_{g,1})$ acts on the space of framings $\mr{Bun}_\partial(\mr{Fr}(TW_{g,1}),\Theta_\mr{fr}; \ell_\partial)$ by taking derivatives. In particular it acts through the topological monoid $\mr{Bun}_{\partial}(\mr{Fr}(TW_{g,1}))$ of $\mr{GL}_{2n}(\bR)$-maps $\mr{Fr}(TW_{g,1}) \to \mr{Fr}(TW_{g,1})$ which are the identity on the boundary. Using the reference framing $\tau$, the argument in \cite[Section 4]{KR-WAlg} provides a homeomorphism of topological monoids
\begin{equation}\label{eq:BunTrivialised}
\mr{Bun}_{\partial}(TW_{g,1}) \overset{\cong}\lra \mr{map}_{\partial}(W_{g,1},W_{g,1} \times \mr{GL}_{2n}(\bR))
\end{equation}
under which composition of bundle maps corresponds to the operation 
\[(f,\lambda) \circledast (g,\rho) \coloneqq (f \circ g,(\lambda \circ g) \cdot \rho),\]
with $\circ$ denoting composition of maps and $\cdot$ denoting pointwise multiplication. The right action of $\mr{Bun}_{\partial}(\mr{Fr}(TW_{g,1}))$ on $\mr{Bun}_\partial(\mr{Fr}(TW_{g,1}),\Theta_\mr{fr}; \ell_\partial)$ by precomposition then corresponds to
\begin{align*}
\mr{map}_\partial(W_{g,1}, \mr{GL}_{2n}(\bR)) \times \mr{map}_{\partial}(W_{g,1},W_{g,1} \times \mr{GL}_{2n}(\bR))  &\lra \mr{map}_\partial(W_{g,1}, \mr{GL}_{2n}(\bR))\\
(h, (f,\lambda)) &\longmapsto  (h \circ f) \cdot \lambda,
\end{align*}
under the homeomorphisms \eqref{eq:FrTRivialised} and \eqref{eq:BunTrivialised}, where $\cdot$ denotes the multiplication in $\mr{GL}_{2n}(\bR)$. We write $h \circledast (f,\lambda)$ for this operation.

Since $\partial W_{g,1} \to W_{g,1}$ is $0$-connected, all maps in the right hand side of \eqref{eq:BunTrivialised} have image in $W_{g,1} \times \mr{GL}^+_{2n}(\bR)$, with $\mr{GL}^+_{2n}(\bR) \leq \mr{GL}_{2n}(\bR)$ the path component of orientation-preserving invertible matrices. As the inclusion $\mr{SO}(2n) \hookrightarrow \mr{GL}^+_{2n}(\bR)$ is a homotopy equivalence, we phrase later computations in terms of the homotopy groups of $\mr{SO}(2n)$ rather than $\mr{GL}^+_{2n}(\bR)$.

\subsection{The moduli space of framed manifolds} \label{sec:moduli-spaces}

 The \emph{moduli space of framed manifolds diffeomorphic to $W_{g,1}$ relative to the boundary}, mentioned in the introduction, is defined to be the homotopy quotient
	\begin{equation}\label{eq:DefFrModuli}
	B\mr{Diff}^\mr{fr}_\partial(W_{g,1};\ell_\partial) \coloneqq \mr{Bun}_\partial(\mr{Fr}(TW_{g,1}),\Theta_\mr{fr}; \ell_\partial) \sslash \mr{Diff}_\partial(W_{g,1}).
	\end{equation}
There is a bijection
	\[\pi_0(B\mr{Diff}^\mr{fr}_\partial(W_{g,1};\ell_\partial)) \overset{\sim}\longleftarrow \mr{Str}^\mr{fr}_\partial(W_{g,1}) / \Gamma_g,\]
	and this is the set which Theorem \ref{thm:main} proposes to describe.
		
	For $[\ell] \in \mr{Str}^\mr{fr}_\partial(W_{g,1})$ we shall write
		\[\Gamma_g^{\fr,[\ell]} \coloneqq \mr{Stab}_{\Gamma_g}([\ell])\]
		for its stabiliser. 
		
		The \emph{framed mapping class group} of a framing $\ell \in \mr{Bun}_\partial(\mr{Fr}(TW_{g,1}),\Theta_\mr{fr}; \ell_\partial)$ is defined as
	\begin{equation*}
	\check{\Gamma}^{\fr, \ell}_g \coloneqq  \pi_1(B\mr{Diff}^\mr{fr}_\partial(W_{g,1};\ell_\partial), \ell),
	\end{equation*}
	and the long exact sequence for the homotopy orbits \eqref{eq:DefFrModuli} gives a surjection $\check{\Gamma}^{\fr, \ell}_g \to {\Gamma}^{\fr, [\ell]}_g$. We write
	\[G_g^{\fr,[\ell]} \coloneqq \mr{im}({\Gamma}^{\fr, [\ell]}_g \to {\Gamma}_g \to G'_g).\]

\subsection{Relaxing the boundary condition}\label{sec:Relaxing} It will be helpful to relax the condition that framings agree with $\ell_\partial$ on all of the boundary and only ask that they agree at a point.

Fix a point $\ast \in \partial W_{g,1}$. Then we let $\mr{Bun}_*(\mr{Fr}(TW_{g,1}),\Theta_\mr{fr}; \ell_*)$ denote the space of $\mr{GL}_{2n}(\bR)$-equivariant maps which agree with $\ell_\partial$ at $\ast \in \partial W_{g,1}$, and just as before we write $\mr{Str}^\mr{fr}_\ast(W_{g,1}) \coloneqq \pi_0(\mr{Bun}_*(\mr{Fr}(TW_{g,1}),\Theta_\mr{fr}; \ell_*))$ for its set of path components. Using the vanishing of Whitehead brackets in $\mr{SO}(2n)$, as in \cite[Section 8.2.2]{KR-WAlg}  we obtain a short exact sequence
\[\begin{tikzcd} 
0 \rar &[-12pt] \mr{Str}^\mr{fr}_\partial(D^{2n}) \dar{\cong} \rar{\circlearrowright} &[-9pt] \mr{Str}^\mr{fr}_\partial(W_{g,1}) \rar \dar{\cong} &[-9pt] \mr{Str}^\mr{fr}_\ast(W_{g,1}) \rar \dar{\cong} &[-12pt] 0 \\
0 \rar & \pi_{2n}(\mr{SO}(2n)) \rar{\circlearrowright} &  \pi_0(\mr{map}_\partial(W_{g,1}, \mr{SO}(2n))) \rar & \mr{Hom}(H_n,\pi_n(\mr{SO}(2n))) \rar & 0.
\end{tikzcd}\]
Here the vertical isomorphisms are induced by $\tau$, and this is in fact a short exact sequence of groups using the group structure coming from pointwise multiplication in $\mr{SO}(2n)$. The groups $\pi_n(\mr{SO}(2n))$ are well-known by Bott periodicity, and the groups $\pi_{2n}(\mr{SO}(2n))$ were computed by Kervaire \cite[p.~161]{KervaireNonstable} (see Table \ref{tab.pi2no}).

\begin{table}[h]
	\centering
	\caption{The groups $\pi_{2n}(\mr{SO}(2n))$ for $n \geq 1$, with the exceptions that $\pi_2(\mr{SO}(2)) = 0$ and $\pi_6(\mr{SO}(6)) = 0$.}
	\label{tab.pi2no}
	\begin{tabular}{lcccccccc}
		\toprule
		$n \pmod 4$ & 0          & 1          & 2 & 3    \\ \midrule
		$\pi_{2n}(\mr{SO}(2n))$  & $(\bZ/2)^3$ & $\bZ/4$ & $ (\bZ/2)^2 $ & $\bZ/4$ \\ \bottomrule
	\end{tabular}
\end{table}

The top short exact sequence is evidently equivariant for the right action of the mapping class group $\Gamma_g$. The induced right actions on the bottom short exact sequence are as follows. The group $\Gamma_g$ acts trivially on the left-hand term. It acts on the middle term via the derivative map $\Gamma_g \to \pi_0(\mr{Bun}_{\partial}(\mr{Fr}(TW_{g,1})))$, the identification
 $\mr{Bun}_{\partial}(\mr{Fr}(TW_{g,1})) \cong \mr{map}_{\partial}(W_{g,1},W_{g,1} \times \mr{GL}_{2n}(\bR))$ from \eqref{eq:BunTrivialised} given by $\tau$, and the action $\circledast$ described above. It acts on the right-hand term via the derivative map composed with the map $\mr{Bun}_{\partial}(\mr{Fr}(TW_{g,1})) \to \mr{Bun}_{\ast}(\mr{Fr}(TW_{g,1}))$ which relaxes the boundary condition to only require the $\mr{GL}_{2n}(\bR)$-maps to be the identity over $* \in \partial W_{g,1}$, followed by the analogue
\[\mr{Bun}_{\ast}(\mr{Fr}(TW_{g,1})) \lra \mr{map}_{\ast}(W_{g,1},W_{g,1} \times \mr{GL}_{2n}(\bR))\]
of \eqref{eq:BunTrivialised} induced by $\tau$, followed by the formula $\alpha \circledast (B, \beta) = \alpha \circ B + \beta$ written in terms of
\[\pi_0(\mr{map}_{\ast}(W_{g,1},W_{g,1} \times \mr{GL}_{2n}(\bR))) \cong \mr{End}(H_n) \ltimes \mr{Hom}(H_n,\pi_n(\mr{SO}(2n))),\]
identifying the right side as a semi-direct product (of monoids). 

The $\Gamma_g$- and $\mr{Str}^\mr{fr}_\partial(D^{2n})$-actions on $\mr{Str}^\mr{fr}_\partial(W_{g,1})$ commute because the $\smash{\mr{Str}^\mr{fr}_\partial(D^{2n})}$-action is by changing the framings in a small disc near the boundary, and diffeomorphisms in $\Gamma_g$ can be changed by an isotopy so that they fix such a disc.

We write $[[\ell]] \in \mr{Str}^\mr{fr}_\ast(W_{g,1})$ for the class of a framing $\ell$, and let
\[{\Gamma}_g^{\fr, [[\ell]]} \coloneqq \mr{Stab}_{\Gamma_g}([[\ell]]).\]
 We define
	\[G_g^{\fr,[[\ell]]} \coloneqq \mr{im}({\Gamma}^{\fr, [[\ell]]}_g \to {\Gamma}_g \to G'_g).\]
	
	\begin{remark}\label{rem:ThetaActsTriv} 
	In \cite{KR-WAlg} we instead defined $G_g^{\fr,[[\ell]]}$ as the image of the stabiliser ${\Lambda}_g^{\fr, [[\ell]]} \coloneqq \mr{Stab}_{\Lambda_g}([[\ell]])$ in $G_g'$. The group $\Lambda_g$ of \cite[Section 3.2]{KR-WAlg} was defined in terms of self-embeddings of $W_{g,1}$, but by the Weiss fiber sequence \cite[Equation (3)]{KR-WAlg} it satisfies $\Lambda_g = \Gamma_g/\Theta_{2n+1}$ so can be discussed without referring to embeddings. For an interpretation of ${\Lambda}_g^{\fr, [[\ell]]}$ in terms of self-embeddings see \cite[Sections 8.1 and 8.5.2]{KR-WAlg}.
	
	The group $\Lambda_g = \Gamma_g/\Theta_{2n+1}$ acts on $\mr{Str}^\mr{fr}_\ast(W_{g,1})$ because $\Theta_{2n+1}$ consists of diffeomorphisms supported in a small disc near the boundary, and when the boundary condition has been relaxed near this ball the derivatives of such diffeomorphisms are homotopic to the identity. (In fact, by Lemma \ref{lem:ThetaActsTriv} the subgroup $\Theta_{2n+1}$ already acts trivially on  $\mr{Str}^\mr{fr}_\partial(W_{g,1})$.) Then ${\Lambda}_g^{\fr, [[\ell]]} = {\Gamma}_g^{\fr, [[\ell]]}/\Theta_{2n+1}$, so the images of ${\Lambda}_g^{\fr, [[\ell]]}$ and ${\Gamma}_g^{\fr, [[\ell]]}$ in $G'_g$ are equal and hence the two definitions of $G_g^{\fr,[[\ell]]}$ agree.
	\end{remark}

\subsection{Quadratic refinements}

The group $H_n = H_n(W_{g,1};\bZ)$ is equipped with the intersection form $\lambda \colon H_n \otimes H_n \to \bZ$, which is $(-1)^n$-symmetric. A function 
\[\mu \colon H_n \lra \begin{cases}
\bZ & \text{ if $n$ is even}\\
 \bZ/2 & \text{ if $n$ is odd}
\end{cases}\]
is called a \emph{quadratic refinement} of $(H_n, \lambda)$ if it satisfies
\begin{equation}\label{eq:QuadProp}
\begin{aligned}
\mu(a \cdot x) &= a^2 \mu(x) \\
\mu(x+y) &= \mu(x) + \mu(y) + \lambda(x,y),
\end{aligned}
\end{equation}
for $a \in \bZ$ and $x, y \in H_n$, where in the latter equation $\lambda(x,y)$ is reduced modulo 2 if $n$ is odd. If $n$ is even then these properties imply that $\mu(x) = \tfrac{1}{2} \lambda(x,x)$, which is well-defined since $(H_n,\lambda)$ is a hyperbolic form and hence even, so $(H_n, \lambda)$ has a unique quadratic refinement and it carries no further information. We shall therefore now suppose that $n$ is odd, and write $\mr{Quad}(H_n, \lambda)$ for the set of quadratic refinements.

If $\mu$ and $\mu'$ are quadratic refinements of $(H_n, \lambda)$ then $\mu'-\mu\colon H_n \to \bZ/2$ is linear, so $\mr{Quad}(H_n, \lambda)$ forms a $\mr{Hom}(H_n, \bZ/2)$-torsor. Equivalently, choosing a symplectic basis $e_1, f_1, e_2, f_2, \ldots, e_g, f_g$ for $H_n$, a quadratic refinement $\mu$ is uniquely and freely determined by the quadratic property \eqref{eq:QuadProp} and the values $\mu(e_1), \ldots, \mu(f_g) \in \bZ/2$.

\subsubsection{Classification of quadratic refinements}\label{sec:ClassQuad}

The action of $\mr{Sp}_{2g}(\bZ) $ on $\mr{Quad}(H_n, \lambda)$ is well-known to have two orbits, distinguished by the \emph{Arf invariant}
\[\mr{Arf}(\mu) = \sum_{i=1}^g \mu(e_i) \mu(f_i) \in \bZ/2,\]
where $e_1, f_1, e_2, f_2, \ldots, e_g, f_g \in H_n$ is a symplectic basis (\cite{Arf}, using Theorem 1 of \cite{NewmanSmart}). 

Let us introduce some further notation for specific quadratic forms. Let $H(0)$ denote the module $\bZ \{e,f\}$ with anti-symmetric form determined by $\lambda(e,f) = 1$ and quadratic refinement determined by $\mu(e) = 0 = \mu(f)$. Let $H(1)$ denote the module $\bZ \{e,f\}$ with the same anti-symmetric form but quadratic refinement determined by $\mu(e) = 1 =\mu(f)$. These quadratic forms have Arf invariant $0$ and $1$ respectively. Since the Arf invariant is additive under orthogonal sum, is valued in $\bZ/2$, and together with rank is a complete invariant of quadratic forms, there is an isomorphism $H(1)^{\oplus 2} \cong H(0)^{\oplus 2}$ of quadratic forms. (In the proof of Lemma \ref{lem:SpaCalc} we will make an explicit choice of such an isomorphism.)

\begin{example} There are $2^{2g-1}+2^{g-1}$ quadratic refinements of Arf invariant 0. The \emph{standard quadratic refinement} $\mu_0$ is that of $H(0)^{\oplus g}$, determined by
\[\mu_0(e_1) = \mu_0(f_1) = \cdots = \mu_0(e_g) = \mu_0(f_g)=0.\]
The group $\mr{Sp}_{2g}^q(\bZ)$ in the statement of Theorem \ref{thm:kreck} is the stabiliser of $\mu_0$ for the action of $Sp_{2g}(\bZ)$ on the set of quadratic refinements.
\end{example}

\begin{example}\label{exam:arf1} Similarly, there are $2^{2g-1}-2^{g-1}$ quadratic refinements of Arf invariant $1$. For concreteness, we take $H(1) \oplus H(0)^{\oplus g-1}$ as the standard choice, and write $\mr{Sp}^a_{2g}(\bZ) \leq \mr{Sp}_{2g}(\bZ)$ for its stabiliser.
\end{example}

\subsubsection{A quadratic form}\label{sec:AQuadForm}

Suppose that $n > 1$. The group $\pi_n(\mr{Fr}(TW_{g,1}))$ may be interpreted via Hirsch--Smale theory as the set ${I}_n^{\fr}(W_{g,1})$ of regular homotopy classes of immersions $j \colon S^n \times D^n \looparrowright W_{g,1}$. Using the map $\pi_n(\mr{Fr}(TW_{g,1})) \to \pi_n(W_{g,1}) = H_n$ and the intersection form $\lambda$, this has a (degenerate) $(-1)^n$-symmetric bilinear form $\lambda^{\fr}$. This has a quadratic refinement $\mu^{\fr}$ given by
\[\mu^{\fr}([j]) = \#\{\text{self-intersections of $j\vert_{S^n \times \{0\}}$}\} \pmod 2.\]
This construction is due to Wall \cite[Theorem 5.2]{wallscm}; see \cite[Definition 5.2]{grwstab1} for a discussion specific to the manifolds $W_{g,1}$.

As the manifolds $W_{g,1}$ admit a framing, there is a splittable short exact sequence
\[0 \lra \pi_n(\mr{SO}(2n)) \overset{i}\lra \pi_n(\mr{Fr}(TW_{g,1})) \lra \pi_n(W_{g,1}) \lra 0.\]
The group $i(\pi_n(\mr{SO}(2n)))$ lies in the radical of the bilinear form $\lambda^{\fr}$, as by definition $\lambda^{\fr}$ factors over $\pi_n(\mr{Fr}(TW_{g,1})) \to \pi_n(W_{g,1}) = H_n$. Thus although the function $\mu^{\fr}$ is quadratic, the composition $\mu^{\fr} \circ i\colon \pi_n(\mr{SO}(2n)) \to \bZ/2$ is linear.

\begin{lemma}\label{lem:Fig8}
The map $\mu^{\fr} \circ i\colon \pi_n(\mr{SO}(2n)) \to \bZ/2$ is onto if and only if $n=3,7$.
\end{lemma}
\begin{proof}
Recall the Whitney ``figure eight'' immersion $S^n \looparrowright D^{2n}$, which has one double point and has normal bundle isomorphic to $TS^n$.

If $n=3,7$ then $TS^n$ is trivial, so the Whitney immersion may be extended to an immersion $j \colon S^n \times D^n \looparrowright D^{2n} \subset W_{g,1}$ which satisfies $\mu([j]) = 1$, and so the composition
	\[\pi_n(\mr{SO}(2n)) = \pi_n(\mr{Fr}(D^{2n})) \overset{i}\lra \pi_n(\mr{Fr}(TW_{g,1})) \overset{\mu^{\fr}}\lra \bZ/2\]
	is surjective.
	
	For $n \neq 3,7$, there are two cases. If $n=2$, then $\pi_n(\mr{SO}(2n)) = 0$ (e.g.~by Bott periodicity) and the result follows. So let us suppose $n \neq 2$, and for a contradiction assume that $j \colon S^n \times D^n \looparrowright D^{2n} \subset W_{g,1}$ has $\mu^{\fr}([j])=1$. We may form the ambient connect-sum inside $D^{2n}$ of the immersion $j\vert_{S^n \times \{0\}}$ with a disjoint copy of the Whitney immersion, giving an immersion $j'\colon S^n \looparrowright D^{2n}$ having an even number of double points and having normal bundle $TS^n$. Using the Whitney trick we can eliminate the double points to obtain an embedding $j''$, still having normal bundle $TS^n$: as $TS^n$ is a non-trivial bundle for $n \neq 3,7$ this is a contradiction, by \cite{KervaireNormal}.
\end{proof}

It follows that for $n \neq 3,7$ the function $\mu^{\fr}$ descends to a function $\mu\colon H_n \to \bZ/2$, a quadratic refinement of $(H_n, \lambda)$. As the standard symplectic basis of $H_n$ may be represented by embedded normally-framed spheres, in the notation of the previous section this gives the quadratic form $H(0)^{\oplus g}$.  Diffeomorphisms of $W_{g,1}$ must also preserve this quadratic refinement: this accounts for the fact that $G'_g = \mr{Sp}_{2g}^q(\bZ)$ in these dimensions in Theorem \ref{thm:kreck}.

\subsubsection{Quadratic refinements from framings}\label{sec:QuadFram}

For $n$ odd a framing $\ell \colon W_{g,1} \to \mr{Fr}(TW_{g,1})$ may be used to define
\[\mu_\ell \colon H_n = \pi_n(W_{g,1}) \overset{\ell_*}\lra \pi_n(\mr{Fr}(TW_{g,1})) \overset{\mu^{\fr}}\lra \bZ/2,\]
which is a quadratic refinement of $(H_n, \lambda)$. This construction provides a $\Gamma_g$-equivariant function
\[\Phi \colon \mr{Str}^\mr{fr}_\ast(W_{g,1}) \lra \mr{Quad}(H_n, \lambda),\]
where the action of $\Gamma_g$ on $\mr{Quad}(H_n, \lambda)$ is via $G'_g$. 

\begin{lemma}\label{lem:QuadStr1}
	For $n=3,7$ the function $\Phi$ is surjective.
\end{lemma}

\begin{proof}
Choose a framing $\ell$, and let $\mu' \in \mr{Quad}(H_n, \lambda)$. The function $\mu'-\mu_\ell$ is a homomorphism $L \colon H_n \to \bZ/2$, and as $H_n$ is a free $\bZ$-module and $\mu^{\fr} \circ i$ is surjective by Lemma \ref{lem:Fig8}, we may choose a homomorphism $\delta \colon H_n \to \pi_n(\mr{SO}(2n))$ such that $\mu^{\fr} \circ i \circ \delta = L$. But then if the framing $\ell$ is changed using $\delta$ to get a new framing $\delta \cdot \ell$, we have $\mu_{\delta \cdot \ell} = \mu_\ell + L$, so $\mu' = \mu_{\delta \cdot \ell}$. Thus $\Phi$ is a surjection.
\end{proof}

\subsection{Orbits and stabilisers}\label{sec:OrbStab}

Our proof of Theorem \ref{thm:main} will be in terms of the sequence
\begin{equation}\label{eq:FundSeq}
0 \to {\Gamma}_{g}^{\fr, [\ell]} \to {\Gamma}_g^{\fr, [[\ell]]} \overset{{f}_\ell}\lra \mr{Str}^\mr{fr}_\partial(D^{2n}) \xrightarrow{- \cdot [\ell]} \mr{Str}^\mr{fr}_\partial(W_{g,1})/\Gamma_g \to \mr{Str}^\mr{fr}_{*}(W_{g,1})/\Gamma_g \to \{*\}
\end{equation}
which is exact in the sense of groups and pointed sets. In particular, $f_\ell$ is a group homomorphism. This sequence comes from the long exact sequence on homotopy groups for the principal $\mr{Str}^\mr{fr}_\partial(D^{2n})$-bundle
\[\mr{Str}^\mr{fr}_\partial(D^{2n}) \lra \mr{Str}^\mr{fr}_\partial(W_{g,1}) \hcoker \Gamma_g \lra \mr{Str}^\mr{fr}_\ast(W_{g,1}) \hcoker \Gamma_g,\]
and the orbit-stabiliser theorem.

The strategy of our argument will be as follows. We will estimate the size of $\mr{im}({f}_\ell) = {\Gamma}_g^{\fr, [[\ell]]}/{\Gamma}_g^{\fr, [\ell]}$ from below using the surjection
\[{\Gamma}_g^{\fr, [[\ell]]}/{\Gamma}_g^{\fr, [\ell]} \lra G_g^{\fr,[[\ell]]}/G_g^{\fr,[\ell]}.\]
We will identify the group $G_g^{\fr,[[\ell]]}$ as the automorphism group of a certain quadratic form (in fact we will have $G_g^{\fr,[[\ell]]} = G'_g$ for $n \neq 1, 3,7$, and a mild variant for $n=1,3,7$), whose abelianisation is known to have order 4. We will then show by geometric considerations that $G_g^{\fr,[\ell]}$ has trivial abelianisation, so that $G_g^{\fr,[[\ell]]}/G_g^{\fr,[\ell]}$ has order at least 4. We then use $\mr{Str}^\mr{fr}_\partial(D^{2n}) = \pi_{2n}(\mr{SO}(2n))$ and consult Table \ref{tab.pi2no}. 

For $n \not\equiv 0 \pmod 4$ the group $\pi_{2n}(\mr{SO}(2n))$ has order 4, so it follows that the map ${f}_\ell$ is a surjection, and so ${\Gamma}_g^{\fr, [\ell]}$ has index precisely 4 in ${\Gamma}_g^{\fr, [[\ell]]}$. Thus $G_g^{\fr,[[\ell]]}/G_g^{\fr,[\ell]}$ has order precisely 4, from which we will deduce that $G_g^{\fr,[\ell]}$ is precisely the commutator subgroup of $G_g^{\fr,[[\ell]]}$.

For $n \equiv 0 \pmod 4$ the group $\pi_{2n}(\mr{SO}(2n))$ has order $8$, and the argument is a little more complicated. We will show that the image of ${f}_\ell$ has index 2 in $\mr{Str}^\mr{fr}_\partial(D^{2n}) = \pi_{2n}(\mr{SO}(2n)) = (\bZ/2)^3$.

\section{Counting framings relative to a point}

 In this section we determine $\mr{Str}^\mr{fr}_\ast(W_{g,1})/\Gamma_g$.

\begin{proposition}\label{prop:upper-bound} Suppose $n \geq 1$ and $g \geq 2$.
	\begin{itemize}
		\item If $n \neq 1, 3,7$, then $\mr{Str}^\mr{fr}_\ast(W_{g,1})/\Gamma_g$ consists of a single element.
		\item If $n = 1, 3,7$, then $\mr{Str}^\mr{fr}_\ast(W_{g,1})/\Gamma_g$ consists of two elements.
	\end{itemize}
	If $n > 1$ then these in fact hold for $g \geq 1$.
\end{proposition}

\subsection{The cases $n \geq 3$} In this case Theorem \ref{thm:kreck} is available, and we will study the action of $I_g \leq \Gamma_g$ on $\mr{Str}^\mr{fr}_\ast(W_{g,1})$. Derivatives of elements of $\Theta_{2n+1} \leq I_g$ are bundle maps supported in a small disc which can be taken to be near the boundary: as in Remark \ref{rem:ThetaActsTriv} these act trivially on $\mr{Str}^\mr{fr}_\ast(W_{g,1})$. Thus the action of $I_g$ on $\mr{Hom}(H_n,\pi_n(\mr{SO}(2n)))$ factors over $I_g \to I_g / \Theta_{2n+1} = \mr{Hom}(H_n,S\pi_n(\mr{SO}(n)))$. This lands in the subgroup 
\[\mr{Hom}(H_n,\pi_n(\mr{SO}(2n))) \subset \mr{GL}(H_n) \ltimes \mr{Hom}(H_n,\pi_n(\mr{SO}(2n)))\]
of the invertible elements of $\pi_0(\mr{map}_{\ast}(W_{g,1},W_{g,1} \times \mr{GL}_{2n}(\bR)))$. The resulting map $\mr{Hom}(H_n,S\pi_n(\mr{SO}(n))) \to \mr{Hom}(H_n,\pi_n(\mr{SO}(2n)))$ is given by composition with the homomorphism $S\pi_n(\mr{SO}(n)) \to \pi_n(\mr{SO}(2n))$. This homomorphism was studied by Levine; Theorem 1.4 of \cite{levineexotic} and Table \ref{tab.spino} imply:

\begin{lemma}\label{lem:levine} For $2n \geq 6$ the stabilisation $S\pi_n(\mr{SO}(n)) \to \pi_n(\mr{SO}(2n))$ is: 
	\begin{enumerate}[\indent (i)]
		\item surjective with kernel $\bZ/2$ when $n$ is even, 
		\item an isomorphism when $n$ is odd but not $3$ or $7$, 
		\item injective with cokernel $\bZ/2$ when $n = 3,7$.
	\end{enumerate}
\end{lemma}

We conclude that:

\begin{proposition}\label{prop:framings-rel-point} Suppose $2n \geq 6$ and consider the action of $I_g$ on $\mr{Str}^\mr{fr}_\ast(W_{g,1})$.
\begin{enumerate}[(i)]
\item When $n \neq 3,7$, this action has a single orbit. 

\item When $n = 3,7$, the set of orbits is in bijection with $\mr{Hom}(H_n, \bZ/2)$. 
\end{enumerate}
In either case the stabiliser $I_g^{\fr, [[\ell]]}$ of any $[[\ell]] \in \mr{Str}^\mr{fr}_\ast(W_{g,1})$ satisfies
\[I_g^{\fr, [[\ell]]}/\Theta_{2n+1} \cong \begin{cases}
0 & \text{ if $n$ is odd}\\
\mr{Hom}(H_n, \bZ/2) &\text{ if $n$ is even}.
\end{cases}\]
\end{proposition}

\begin{proof}When $n \neq 3,7$, Lemma \ref{lem:levine} says that $I_g$ surjects on to $\mr{Hom}(H_n,\pi_n(\mr{SO}(2n)))$. The action on $\mr{Hom}(H_n,\pi_n(\mr{SO}(2n)))$ is through addition, which is therefore transitive. When $n=3,7$ the lemma shows that $I_g$ maps to $\mr{Hom}(H_n,\pi_n(\mr{SO}(2n)))$ with cokernel $\mr{Hom}(H_n, \bZ/2)$, so the orbits are in bijection with this set. For the stabiliser, Lemma \ref{lem:levine} shows that the kernel of $S\pi_n(\mr{SO}(n)) \to \pi_n(\mr{SO}(2n))$ is $0$ if $n$ is odd and $\bZ/2$ if $n$ is even.
\end{proof}

This proves Proposition \ref{prop:upper-bound} when $n \geq 3$ and $n \neq 3,7$, as $\mr{Str}^\mr{fr}_\ast(W_{g,1})/I_g$ already consists of a single point, so $\mr{Str}^\mr{fr}_\ast(W_{g,1})/\Gamma_g$ does too. 

To finish the argument in the cases $n=3,7$, we use quadratic refinements to give a more invariant description of $\mr{Str}^\mr{fr}_\ast(W_{g,1})/I_g$. Recall that in Section \ref{sec:QuadFram} we described a $\Gamma_g$-equivariant function $\Phi \colon \mr{Str}^\mr{fr}_\ast(W_{g,1}) \to \mr{Quad}(H_n, \lambda)$.

\begin{lemma}\label{lem:QuadStr}
	For $n=3,7$ the induced function
	\[\mr{Str}^\mr{fr}_\ast(W_{g,1}) / I_g \lra \mr{Quad}(H_n, \lambda)\]
	is a bijection.
\end{lemma}
\begin{proof}
By Lemma \ref{lem:QuadStr1} it is surjective. The target has $2^{2g}$ elements as it is a $\mr{Hom}(H_n, \bZ/2)$-torsor, and the domain has $2^{2g}$ elements by Proposition \ref{prop:framings-rel-point} (ii), so it is a bijection.
\end{proof}

Referring to the discussion in Section \ref{sec:ClassQuad}, it follows from the theorem of Arf that $\mr{Str}^\mr{fr}_\ast(W_{g,1}) / \Gamma_g$ consists of two elements, distinguished by the Arf invariants of their associated quadratic forms. This completes the proof of Proposition \ref{prop:upper-bound} in the cases $n=3,7$.

\subsection{The case $n=2$}\label{sec:RelPtNEq2}

When $n=2$ we have $\pi_n(\mr{SO}(2n))=0$ (e.g.~by Bott periodicity) so $\mr{Str}^\mr{fr}_\ast(W_{g,1})$ consists of a single point, so has a single $\Gamma_g$-orbit too. 

\subsection{The case $n=1$}\label{sec:RelPtNEq1}

As $\mr{Str}^{\fr}_\partial(D^2) = \pi_2(\mr{SO}(2))=0$, the orbit-stabiliser sequence \eqref{eq:FundSeq} gives a bijection $\mr{Str}^\mr{fr}_{*}(W_{g,1})/\Gamma_g \cong \mr{Str}^\mr{fr}_\partial(W_{g,1})/\Gamma_g$. For $g \geq 2$ the latter is described in \cite[Theorem 2.9]{RWFramed} (for framings one takes $r=0$) as having two elements, which proves Proposition \ref{prop:upper-bound} in this case. 

In fact the analogue of Lemma \ref{lem:QuadStr} also holds in this case. A framing of $W_{g,1}$ determines a Spin structure, which via a construction of Johnson \cite[Section 4]{JohnsonSpin} gives a quadratic refinement of the symplectic form $(H_1(W_{g,1};\bZ), \lambda)$. This construction yields a surjective map
	\[\mr{Str}^\mr{fr}_\ast(W_{g,1}) \lra \mr{Quad}(H_1, \lambda)\]
	which as long as $g \geq 2$ becomes, just as in Lemma \ref{lem:QuadStr}, a bijection after dividing out the action of the Torelli group $I_g$. This has been shown in \cite[Proposition 5.1]{CalderonSalter}. It may also be seen using the methods of \cite[Sections 2.3, 2.4]{RWFramed} and the fact that the Torelli group is generated by Dehn twists along separating curves and bounding pairs of curves.

\subsection{The group $G_g^{\fr, [[\ell]]}$}\label{sec:StabGammaDoubleBracketEll}

The results of the previous sections combine to give the following complete description of the group $G_g^{\fr, [[\ell]]}$.

\begin{proposition}\label{prop:IdentifyingDoubleBracketEll}
For $n \geq 1$, and $g \geq 2$ if $n=1$, we have
\[G_g^{\fr, [[\ell]]} = \begin{cases} \mr{Sp}_{2g}^{q\text{ or }a}(\bZ) & \text{if $n$ is 1, 3 or 7, and $\ell$ has Arf invariant $0$ or $1$,} \\
	\mr{Sp}_{2g}^q(\bZ) & \text{if $n$ is odd but not 1, 3 or 7,} \\
	\mr{O}_{g,g}(\bZ) & \text{if $n$ is even.}\end{cases}\]
\end{proposition}
\begin{proof}
If $n=2$ then as in Section \ref{sec:RelPtNEq2} there is only one framing relative to a point, so $\Gamma_g^{\fr, [[\ell]]} = \Gamma_g$ and hence $G_g^{\fr, [[\ell]]} = G'_g$.

For $n \neq 2$ the group $G_g^{\fr, [[\ell]]}$ is the stabiliser of the class of $\ell$ in $\mr{Str}^\mr{fr}_\ast(W_{g,1}) / I_g$ under the residual $\Gamma_g/I_g = G'_g$-action. If $n \neq 1, 3,7$ then this set has a single element by Proposition \ref{prop:framings-rel-point} (i) and so the stabiliser is $G'_g$ itself. If $n=1, 3,7$ then this set is identified with $\mr{Quad}(H_n, \lambda)$ in Lemma \ref{lem:QuadStr}, so $G_g^{\fr, [[\ell]]}$ is the stabiliser of the quadratic form determined by $\ell$. As there are two orbits of quadratic forms, distingushed by their Arf invariant, this stabiliser is conjugate to $\mr{Sp}^q_{2g}(\bZ)$ if the Arf invariant is 0, and to $\mr{Sp}^a_{2g}(\bZ)$ if the Arf invariant is 1.
\end{proof}

For $2n \geq 6$ there is by definition an extension
\[0 \lra I_g^{\fr, [[\ell]]} \lra \Gamma_g^{\fr, [[\ell]]} \lra G_g^{\fr, [[\ell]]} \lra 0\]
and in Proposition \ref{prop:framings-rel-point} we have shown that there is an extension
\[0 \lra \Theta_{2n+1} \lra  I_g^{\fr, [[\ell]]} \lra \begin{cases}
0 & \text{ if $n$ is odd}\\
\mr{Hom}(H_n, \bZ/2) &\text{ if $n$ is even}
\end{cases} \lra 0.\]
With Proposition \ref{prop:IdentifyingDoubleBracketEll} this gives a description of $\Gamma_g^{\fr, [[\ell]]}$ analogous to the theorem of Kreck (Theorem \ref{thm:kreck}).

\section{Arithmetic groups} 

\subsection{Abelianisations of some arithmetic groups}\label{sec:Abs}

By stabilising by direct sum with a hyperbolic form, of Arf invariant 0 in the case of quadratic structures, we have stable groups
\[\mr{Sp}_\infty(\bZ), \quad \mr{Sp}^q_\infty(\bZ), \quad \mr{Sp}^a_\infty(\bZ), \quad \mr{O}_{\infty, \infty}(\bZ).\]
Furthermore, as $H(0)^{\oplus 2} \cong H(1)^{\oplus 2}$ there is a direct system of groups containing both $\{\mr{Sp}^q_{2g}(\bZ)\}_{g \geq 0}$ and $\{\mr{Sp}^a_{2g}(\bZ)\}_{g \geq 0}$ cofinally, so $\mr{Sp}^q_\infty(\bZ) \cong \mr{Sp}^a_\infty(\bZ)$. The abelianisations of these stable groups are well-known:
\begin{align*}
H_1(\mr{Sp}_\infty(\bZ)) &= 0,\\
H_1(\mr{Sp}^q_\infty(\bZ)) &= H_1(\mr{Sp}^a_\infty(\bZ)) = \bZ/4,\\
H_1(\mr{O}_{\infty, \infty}(\bZ)) &= (\bZ/2)^2.
\end{align*}
These are collected from the literature in \cite[Proposition 2.2]{grwabelian}. Such automorphism groups of quadratic forms over $\bZ$ enjoy homological stability: in the generality needed here this may be found in \cite[Theorem 3.25]{Friedrich}, but for some of these groups it was known much earlier. We shall only need to know that the abelianisation of the $g=1$ group surjects onto to the abelianisation of the $g=\infty$ one, but we give complete information about their abelianisations for all $g$ in Table \ref{tab.h1}.

\begin{table}[h]
	\centering
	\caption{The first homology groups of $\mr{Sp}_{2g}(\bZ)$, $\mr{Sp}^q_{2g}(\bZ)$, $\mr{Sp}^a_{2g}(\bZ)$, and $\mr{O}_{g,g}(\bZ)$.}
	\label{tab.h1}
	\begin{tabular}{lcccccccc}
		\toprule
		$g$ & 1        & 2         & $\geq 3$     \\ \midrule
		$H_1(\mr{Sp}_{2g}(\bZ))$  & $\bZ/12$ & $\bZ/2$ & $0$ \\
		$H_1(\mr{Sp}^q_{2g}(\bZ))$  & $\bZ/4 \oplus \bZ$ & $\bZ/4 \oplus \bZ/2$ & $\bZ/4$ \\ 
				$H_1(\mr{Sp}^a_{2g}(\bZ))$  & $\bZ/12$ & $\bZ/4$ & $\bZ/4$ \\ 
		$H_1(\mr{O}_{g,g}(\bZ))$  & $(\bZ/2)^2$ & $(\bZ/2)^3$ & $(\bZ/2)^2$ \\ \bottomrule
	\end{tabular}
\end{table}

\begin{remark}It may be helpful to alert the reader that \cite[page 645]{kreckisotopy} states incorrectly that $\mr{O}_{1,1}(\bZ) = \bZ/4$, a mistake going back to \cite{Sato}.
\end{remark}

\subsubsection{The (quadratic) symplectic groups} 

The first homology group of $\mr{Sp}_{2g}(\bZ)$ and $\mr{Sq}^q_{2g}(\bZ)$ in low genus is tabulated in \cite[Lemma A.1]{KrannichMCG}, and that of $\mr{Sp}^a_{2g}(\bZ)$ has recently been calculated by Sierra \cite{SierraSpin} and will appear in his forthcoming Cambridge PhD thesis. They are as shown in Table \ref{tab.h1}. 

\begin{lemma}\label{lem:SpaCalc}
Each of the stabilisation maps $H_1(\mr{Sp}_{2}(\bZ)) \to H_1(\mr{Sp}_\infty(\bZ))$, $H_1(\mr{Sp}^q_{2}(\bZ)) \to H_1(\mr{Sp}^q_\infty(\bZ))$, and $H_1(\mr{Sp}^a_{2}(\bZ)) \to H_1(\mr{Sp}^a_\infty(\bZ))$ is surjective.
\end{lemma}
\begin{proof}
In the first case there is nothing to show as the stable homology is trivial. The second case is \cite[Lemma A.1 (iii)]{KrannichMCG}.

In the third case observe that as in Example \ref{exam:arf1} the number of quadratic refinements of $(H_n(W_{1,1};\bZ), \lambda)$ having Arf invariant 1 is $2^{2-1}-2^{1-1} = 1$, so we have $\mr{Sp}^a_2(\bZ) = \mr{Sp}_2(\bZ) = \mr{SL}_2(\bZ)$. The first homology of this group is $\bZ/12$ and an element of order 4 is represented by the matrix
\[S \coloneqq \begin{bmatrix} 0 & -1 \\
1 & 0 \end{bmatrix}.\]
Stabilising $H(1)$ by taking the direct sum with another copy of $H(1)$, we get a stabilisation map $\mr{Sp}^a_2(\bZ) \to \mr{Sp}^q_4(\bZ)$. We shall compute its image under the stable abelianisation map $A \colon \mr{Sp}^q_4(\bZ) \to \bZ/4$ using the work of Johnson and Millson \cite{JohnsonMillson} (they denote $A$ by $\lambda$).

To do so, we need to make explicit the isomorphism $H(1) \oplus H(1) \cong H(0) \oplus H(0)$. Let $e_1,f_1,e_2,f_2$ be the standard ``hyperbolic'' basis of $H(1) \oplus H(1)$. Making explicit this isomorphism amounts to finding a basis $\tilde{e}_1,\tilde{f}_1,\tilde{e}_2,\tilde{f}_2$ of $H(1) \oplus H(1)$ satisfying 
\begin{itemize}
	\item $\lambda(\tilde{e}_i,\tilde{e}_j) = 0 = \lambda(\tilde{f}_i,\tilde{f}_j)$,
	\item $\lambda(\tilde{e}_i,\tilde{f}_j) = \delta_{ij}$, and
	\item $\mu(\tilde{e}_i) = \mu(\tilde{f}_i) = 0$.
\end{itemize}
The choice
\[\tilde{e}_1 = e_1+e_2 \qquad \tilde{f}_1 = f_1+e_1+e_2 \qquad \tilde{e}_2 = e_2-f_1+f_2 \qquad \tilde{f}_2 = -f_1+f_2\]
will do. Writing $S \oplus \begin{bsmallmatrix} 1 & 0 \\
0 & 1 \end{bsmallmatrix}$ in terms of the basis $\{\tilde{e}_1,\tilde{e}_2,\tilde{f}_1,\tilde{f}_2\}$ gives the matrix
\[\tilde{S} = \begin{bmatrix} 1 & 1 & -1 & -1 \\ -2 & 2 & 1 & -1 \\ 2 & 0 & -1 & 0 \\ -2 & 1 & 1 & 0\end{bmatrix}.\]

Johnson and Millson give an explicit formula for the stable abelianisation map $A \colon \mr{Sp}^q_4(\bZ) \to \bZ/4 = \{1,-1,i,-i\}$ \cite[pages 147-148]{JohnsonMillson} (their conventions are the reason for reordering our basis). Evaluated on $\tilde{S}$ we are in the case ``2 (ii)'', and $A(\tilde{S}) = i^{-1} \epsilon(1\cdot 1-(-2) \cdot 1) = i$, which has order 4.
\end{proof}

\subsubsection{The orthogonal groups} We provide the analogue of the results of the last section for the groups $\mr{O}_{g,g}(\bZ)$. 

\begin{lemma}\label{lem:h1-gg-improv-even} 
For $g \neq 2$ the first homology groups of $\mr{O}_{g,g}(\bZ)$ are as in Table \ref{tab.h1}. The stabilisation map $H_1(\mr{O}_{1,1}(\bZ)) \to H_1(\mr{O}_{\infty,\infty}(\bZ))$ is a surjection.
\end{lemma}
\begin{proof} 
We shall use results of Hahn--O'Meara: \cite[Theorem 9.2.8]{HahnOMeara} says the kernel $\mr{O}'_{g,g}(\bZ)$ of the map $\mr{O}_{g,g}(\bZ) \to (\bZ/2)^2$ given by the determinant and spinor norm is equal to the subgroup generated by elementary matrices for $g \geq 2$, c.f.~\cite[p.~232]{HahnOMeara}. (The subgroup $\mr{O}'_{g,g}(\bZ)$ is defined in \cite[Section 6.4C]{HahnOMeara} as the kernel of $(\Theta,R)$ with $\Theta$ the spinor norm  \cite[6.4A]{HahnOMeara} and $R$ the residue homomorphism \cite[p.~346]{HahnOMeara}, which is equal to the determinant for the ring $\bZ$ as $\chi(\bZ) = 0$.) By \cite[5.3.8]{HahnOMeara} this subgroup is perfect for $g \geq 3$. It follows that $(\bZ/2)^2$ is the abelianisation for $g \geq 3$.

For $g=1$ it is easy to verify that the determinant and spinor norm map $\mr{O}_{1,1}(\bZ) \to (\bZ/2)^2$ is an isomorphism of groups, and the claim about the stabilisation map follows from this.
\end{proof}

To complete Table \ref{tab.h1} it remains to describe the case $g=2$, though this is not used in the proof of Theorem \ref{thm:main}.

\begin{lemma} 
$H_1(\mr{O}_{2,2}(\bZ)) \cong (\bZ/2)^3$.
\end{lemma}

\begin{proof}
We claim that $\mr{O}'_{2,2}(\bZ) \coloneqq \mr{ker}(\det \oplus \mr{spin}\colon \mr{O}_{2,2}(\bZ) \to (\bZ/2)^2)$ is isomorphic to the group
\[\mr{SL}_2(\bZ) \times \mr{SL}_2(\bZ) / \left\langle \left(\begin{bsmallmatrix} -1 & 0 \\
	0 & -1 \end{bsmallmatrix}, \begin{bsmallmatrix} -1 & 0 \\
	0 & -1 \end{bsmallmatrix} \right)\right\rangle.\]
	By \cite[Theorem 7.2.21]{HahnOMeara} there is an extension
	\[1 \lra \bZ^\times \lra \mr{Spin}_{2,2}(\bZ) \lra \mr{O}'_{2,2}(\bZ) \lra 1,\]
	and by \cite[p.\ 434]{HahnOMeara} there is an exceptional isomorphism
	\[\mr{Spin}_{2,2}(\bZ) \cong \mr{SL}_2(\bZ) \times \mr{SL}_2(\bZ).\]
	
	We can make this explicit as follows. Consider the set $M_{2,2}(\bZ)$ of $(2 \times 2)$-matrices over $\bZ$, equipped with the bilinear form given by 
	\[\langle X,Y\rangle = \mr{tr}(X \Omega Y^t \Omega^t), \qquad \Omega = \begin{bmatrix} 0 & -1 \\
	1 & 0 \end{bmatrix}.\]
Explicitly it is given by
\[\left \langle \begin{bmatrix} a & b \\
c & d \end{bmatrix}, \begin{bmatrix} a' & b' \\
c' & d' \end{bmatrix} \right \rangle = ad'-bc'-cb'+da',\]
so is symmetric and even. This formula also makes it clear that
\[e_1 = \begin{bmatrix} 1 & 0 \\
0 & 0 \end{bmatrix}, \quad f_1 = \begin{bmatrix} 0 & 0 \\
0 & 1\end{bmatrix}, \quad e_2 = \begin{bmatrix} 0 & 1 \\
0 & 0 \end{bmatrix}, \quad f_2 = \begin{bmatrix} 0 & 0 \\-1 & 0\end{bmatrix}\]
provides a hyperbolic basis for $(M_{2,2}(\bZ), \langle-,-\rangle)$. There is a left action of $\mr{SL}_2(\bZ) \times \mr{SL}_2(\bZ)$ on $M_{2,2}(\bZ)$ by $(A,B) \cdot X = AXB^{-1}$, and one may check that this action preserves the form $\langle -,-\rangle$, using that $A^{-1} = \Omega A^t \Omega^t$ for $A \in \mr{SL}_2(\bZ)$. This describes the composition
$$\mr{SL}_2(\bZ) \times \mr{SL}_2(\bZ) \cong \mr{Spin}_{2,2}(\bZ) \lra \mr{O}_{2,2}(\bZ).$$
	
	We see that $\bZ^\times \to \mr{Spin}_{2,2}(\bZ)$ is given by $\left(\begin{bsmallmatrix} -1 & 0 \\
	0 & -1 \end{bsmallmatrix}, \begin{bsmallmatrix} -1 & 0 \\
	0 & -1 \end{bsmallmatrix} \right)$, which establishes the claim, and also that the elements $\left(\begin{bsmallmatrix} 1 & 1 \\
	0 & 1 \end{bsmallmatrix}, \begin{bsmallmatrix} 1 & 0 \\
	0 & 1 \end{bsmallmatrix} \right)$ and $\left(\begin{bsmallmatrix} 1 & 0 \\
	0 & 1 \end{bsmallmatrix}, \begin{bsmallmatrix} 1 & 1 \\
	0 & 1 \end{bsmallmatrix} \right)$ in $\mr{SL}_2(\bZ) \times \mr{SL}_2(\bZ)$ map to
	\[T_1 \coloneqq \begin{bmatrix}
1 & 0 & 0 & -1 \\
0 & 1 & 0 & 0 \\
0 & 1 & 1 & 0 \\
0 & 0 & 0 & 1
\end{bmatrix} \quad \text{ and } \quad T_2 \coloneqq\begin{bmatrix}
1 & 0 & 0 & 0 \\
0 & 1 & 0 & 1 \\
-1 & 0 & 1 & 0 \\
0 & 0 & 0 & 1
\end{bmatrix} \quad \text{respectively,}\]
with respect to the basis $(e_1, f_1, e_2, f_2)$.

We now calculate as follows. It is well known that $H_1(\mr{SL}_2(\bZ)) = \bZ/12$ is generated by $T \coloneqq \begin{bsmallmatrix} 1 & 1 \\
	0 & 1 \end{bsmallmatrix}$, and minus the identity matrix represents the element of order 2 in this group, so $H_1(\mr{O}'_{2,2}(\bZ))$ has a presentation as an abelian group by $\langle T_1, T_2 \, | \, 12 T_1, 12 T_2,  6(T_1+T_2)\rangle$,
	
	The group $\mr{O}_{1,1}(\bZ) = \left\langle\begin{bsmallmatrix} -1 & 0 \\
	0 & -1 \end{bsmallmatrix}, \begin{bsmallmatrix} 0 & 1 \\
	1 & 0 \end{bsmallmatrix} \right\rangle$ may be included in $\mr{O}_{2,2}(\bZ)$  by stabilisation, and is mapped isomorphically to $(\bZ/2)^2$ by the determinant and spinor norm. Thus the outer action of $(\bZ/2)^2$ on $\mr{O}'_{2,2}(\bZ)$ may be described by conjugating by (the stabilisations of) these two matrices. Conjugating by $\begin{bsmallmatrix} -1 & 0 \\
	0 & -1 \end{bsmallmatrix} \oplus \begin{bsmallmatrix} 1 & 0 \\
	0 & 1 \end{bsmallmatrix}$ acts as
	\[T_1 \mapsto T_1^{-1} \quad\quad T_2 \mapsto T_2^{-1}\]
	and conjugating by $\begin{bsmallmatrix} 0 & 1 \\
	1 & 0 \end{bsmallmatrix} \oplus \begin{bsmallmatrix} 1 & 0 \\
	0 & 1 \end{bsmallmatrix}$ acts as
	\[T_1 \mapsto T_2^{-1} \quad\quad T_2 \mapsto T_1^{-1}.\]
	Thus the coinvariants of the $(\bZ/2)^2$-action on $H_1(\mr{O}'_{2,2}(\bZ))$ are given by the abelian group presentation
	\[\langle T_1, T_2 \, | \, 12 T_1, 12 T_2,  6(T_1+T_2), 2T_1, 2T_2, T_1+T_2\rangle\]
	which simplifies to give $\bZ/2$. The argument is completed by considering the Lyndon--Hochschild--Serre spectral sequence for $1 \to \mr{O}'_{2,2}(\bZ) \to \mr{O}_{2,2}(\bZ) \to (\bZ/2)^2 \to 1$, which is split by the inclusion of $\mr{O}_{1,1}(\bZ) \cong (\bZ/2)^2$ into $\mr{O}_{2,2}(\bZ)$.
\end{proof}

\subsection{A vanishing result} 

The previous section computed the abelianisation of the arithmetic groups $G'_g$ and $G^{\mr{fr},[[\ell]]}_g$ arising in the discussion in Section \ref{sec:General}. Here we compute the abelianisation of the remaining arithmetic group $G^{\mr{fr},[\ell]}_g$, at least for $n \neq 1, 3$ and in the limit as $g$ tends to infinity. The strategy is quite different: the group $G^{\mr{fr},[\ell]}_g$ arises as a quotient of the framed mapping class group $\check{\Gamma}^{\fr, \ell}_g$, and the proof is based on geometric considerations of framed fibre bundles.

\begin{lemma}\label{lem:triv-fr-stable}
$H_1(G^{\mr{fr},[\ell]}_\infty) =0$ for $n \neq 1, 3$.
\end{lemma}

It will follow by combining Corollary \ref{cor:IdentifyingSingleBracketEll}, Proposition \ref{prop:IdentifyingDoubleBracketEll}, and Table \ref{tab.h1} that this vanishing result does not hold for $n=1,3$.

\begin{proof}
Recall that we write $\check{\Gamma}_g^{\mr{fr},\ell} \coloneqq \pi_1(B\mr{Diff}^\mr{fr}_\partial(W_{g,1};\ell_\partial), \ell)$, so that ${\Gamma}_g^{\mr{fr},[\ell]}$ is the image of the forgetful map $\check{\Gamma}_g^{\mr{fr},\ell} \to \Gamma_g$. The composition
\begin{equation}\label{eq:triv-fr-stable}
\check{\Gamma}_0^{\mr{fr},\ell} = H_1(\check{\Gamma}_0^{\mr{fr},\ell}) \lra H_1(\check{\Gamma}_g^{\mr{fr},\ell}) \lra H_1({\Gamma}_g^{\mr{fr},[\ell]}) \lra H_1(G_g^{\mr{fr},[\ell]})
\end{equation}
is zero, as diffeomorphisms supported inside a disc act trivially on the middle homology of $W_{g,1}$. The two rightmost maps are surjective, as the underlying maps of groups are surjective. 

Let $B\mr{Diff}^\mr{fr}_\partial(W_{g,1};\ell_\partial)_{\ell}$ denote the path component of $B\mr{Diff}^\mr{fr}_\partial(W_{g,1};\ell_\partial)$ containing the framing $\ell$. Recall the notation $G \coloneqq \mr{colim}_{n \to \infty} \mr{hAut}_*(S^n)$ (unrelated to $G_g$ and its variants, but alas standard). By an application of \cite[Theorem 1.5]{grwstab2} there is a map
\[H_1(\check{\Gamma}_g^{\mr{fr},\ell}) = H_1(B\mr{Diff}^\mr{fr}_\partial(W_{g,1};\ell_\partial)_{\ell}) \lra H_1(\Omega^\infty_0 \mathbf{S}^{-2n}) = \pi_{2n+1}^s = \pi_{2n+1}(G)\]
which is an isomorphism in the limit $g \to \infty$ (this formulation is valid even if $2n=4$, as it does not rely on homological stability). Considering the target as framed cobordism, this map is given by the Pontrjagin--Thom construction.

If $n=2$ then we use that $\pi_5^s=0$, so that $\colim_{g \to \infty} H_1(\check{\Gamma}_g^{\mr{fr},\ell}) = 0$ surjects onto $\colim_{g \to \infty} H_1(G_g^{\mr{fr},[\ell]})$.

For $n > 3$ will show that $H_1(\check{\Gamma}_0^{\mr{fr},\ell}) \to \colim_{g \to \infty}H_1(\check{\Gamma}_g^{\mr{fr},\ell})$ is surjective, which with the fact that the composition \eqref{eq:triv-fr-stable} is zero gives the result. To do so we use smoothing theory to obtain an identification $\check{\Gamma}_0^{\mr{fr},\ell} = \pi_{2n+1}(\mr{Top}(2n))$, and it is a matter of interpreting the Pontrjagin--Thom construction to see that $\check{\Gamma}_0^{\mr{fr},\ell} \to \colim_{g \to \infty} H_1(\check{\Gamma}_g^{\mr{fr},\ell}) = \pi_{2n+1}^s = \pi_{2n+1}(G)$ agrees with the natural composition
\[\pi_{2n+1}(\mr{Top}(2n)) \lra \pi_{2n+1}(\mr{Top}) \lra \pi_{2n+1}(G).\]
The homotopy groups $\pi_i(G/\mr{Top})$ are identified with the simply-connected surgery obstruction groups $L_i(\bZ)$ \cite[p.\ 274]{kirbysiebenmann}, which vanish in odd degrees so it follows that the right map is surjective; thus we shall be done if we show that the left map is surjective.

To show this we will instead show that the connecting map
\[\partial \colon \pi_{2n+1}(\mr{Top}, \mr{Top}(2n)) \lra \pi_{2n}(\mr{Top}(2n))\]
is injective. From \cite[page 246]{kirbysiebenmann}, it follows that the map
\[\pi_{2n+1}(\mr{O}, \mr{O}(2n)) \lra \pi_{2n+1}(\mr{Top}, \mr{Top}(2n))\]
is an isomorphism, and as in \cite[Lemma 5.2]{grwabelian} by work of Paechter \cite{Paechter} (the relevant result can be found on p.~249, by taking $p=1$, $m \geq 3$, and $k=4s$ or $4s+2$) we have
\[\pi_{2n+1}(\mr{O}, \mr{O}(2n)) = \begin{cases}
\bZ/4 & n \text{ odd}\\
(\bZ/2)^2 & n \text{ even}.
\end{cases}\]
We therefore consider the diagram 
\[\begin{tikzcd}
& \pi_{2n+1}(\mr{Top}(2n)/\mr{O}(2n)) \dar \rar{\cong}& \pi_0(\mr{Diff}_\partial(D^{2n})) \\
\pi_{2n+1}(\mr{O}, \mr{O}(2n)) \rar \dar[swap]{\cong}& \pi_{2n}(\mr{O}(2n)) \dar \rar & \pi_{2n}(\mr{O})\\
\pi_{2n+1}(\mr{Top}, \mr{Top}(2n)) \rar& \pi_{2n}(\mr{Top}(2n)).
\end{tikzcd}\]
The top vertical map is zero: it corresponds to  the map sending a diffeomorphism to its derivative, which is trivial by Lemma \ref{lem:ThetaActsTriv}. Thus the bottom middle vertical map is injective. On the other hand we know the groups $\pi_{2n}(\mr{O}(2n))$ from Table \ref{tab.pi2no}, and the groups $\pi_{2n}(\mr{O})$ from Bott periodicity, so we may simply check that $\pi_{2n+1}(\mr{O}, \mr{O}(2n)) \to \pi_{2n}(\mr{O}(2n))$ must be injective for $n > 3$. It follows from commutativity of the square that the bottom map is injective, as required.
\end{proof}

\section{Comparing stabilisers and the proof of Theorem \ref{thm:main}}

In this section we wish to analyse the exact sequence
\begin{equation}\label{eq:HatF}
0 \lra {\Gamma}_{g}^{\fr, [\ell]} \lra {\Gamma}_g^{\fr, [[\ell]]} \overset{{f}_\ell}\lra \mr{Str}^\mr{fr}_\partial(D^{2n}) \cong \pi_{2n}(\mr{SO}(2n)),
\end{equation}
coming from \eqref{eq:FundSeq}.

\begin{proposition}\label{prop:FEllAnalysis}
For any $n > 0$ the map ${f}_\ell$ factors as
\[{\Gamma}_g^{\fr, [[\ell]]} \lra {G}_g^{\fr, [[\ell]]} \lra H_1({G}_g^{\fr, [[\ell]]}) \lra H_1({G}_\infty^{\fr, [[\ell]]}) \overset{h_\ell}\lra \pi_{2n}(\mr{SO}(2n))\]
for some $h_\ell$, where the first three maps are the natural quotient, abelianisation, and stabilisation maps.

\begin{enumerate}[(i)]
\item If $n=1,3$ then $\pi_{2n}(\mr{SO}(2n))=0$ (so the map ${f}_\ell$ is surjective).

\item If $n \not\equiv 0 \pmod 4$ and $n \not=1,3$ then $h_\ell$ is an isomorphism.

\item If $n \equiv 0 \pmod 4$ then $h_\ell$ is injective with image of index 2 in $\pi_{2n}(\mr{SO}(2n)) \cong (\bZ/2)^3$.
\end{enumerate}
\end{proposition}

In Remark \ref{rem:index-4-ident} we will determine the index 2 subgroup in part (iii).

\begin{proof}
If $n=1,3$ then $\pi_{2n}(\mr{SO}(2n))=0$ so the claims are vacuous.

By naturality of the sequence \eqref{eq:HatF} under boundary connect-sum, the connecting map ${f}_\ell$ factors over (the abelianisation of) its stabilisation, i.e.\ as
\[{\Gamma}_g^{\fr, [[\ell]]} \lra {\Gamma}_\infty^{\fr, [[\ell]]} \lra H_1({\Gamma}_\infty^{\fr, [[\ell]]}) \overset{g_\ell}\lra \pi_{2n}(\mr{SO}(2n)).\]
To prove the first part we must show that this map $g_\ell$ factors as
\[H_1({\Gamma}_\infty^{\fr, [[\ell]]}) \lra H_1({G}_\infty^{\fr, [[\ell]]}) \overset{h_\ell}\lra \pi_{2n}(\mr{SO}(2n))\]
for some (unique, as the first map is surjective) map $h_\ell$.

If $n=2$ then by \cite[Theorem 1]{kreckisotopy} the map ${\Gamma}_g^{\fr, [[\ell]]} \to {G}_g^{\fr, [[\ell]]}$ has kernel consisting of isotopy classes of diffeomorphisms which are pseudoisotopic to the identity, and by \cite[Theorem 1.4]{quinnisotopy} such a diffeomorphism becomes isotopic to the identity after sufficiently-many stabilisations by $S^2 \times S^2$. Thus these maps become isomorphisms of groups in the colimit, and so in particular $H_1({\Gamma}_\infty^{\fr, [[\ell]]}) \to H_1({G}_\infty^{\fr, [[\ell]]})$ is an isomorphism, so $g_\ell$ factors as desired.

Suppose now that $n >3$. By stabilising if necessary we may suppose that $g$ is large. By Lemma \ref{lem:ThetaActsTriv} the subgroup $\Theta_{2n+1} \leq \Gamma_g$ acts trivially on the set of framings relative to the boundary, so $\Theta_{2n+1} \leq \Gamma_g^{\fr, [\ell]} \leq \Gamma_g^{\fr, [[\ell]]}$ and is therefore annihilated by ${f}_\ell$. Now ${f}_\ell$ is a homomorphism to an abelian group, so factors over $H_1({\Gamma}_g^{\fr, [[\ell]]}/\Theta_{2n+1})$. To calculate the latter group, we use the extension
\[0 \lra \begin{cases}
0 & \text{ $n$ odd}\\
\mr{Hom}(H_n, \bZ/2) &\text{ $n$ even}
\end{cases} \lra {\Gamma}_g^{\fr, [[\ell]]}/\Theta_{2n+1} \lra G_g^{\fr,[[\ell]]} \lra 0\]
from Section \ref{sec:StabGammaDoubleBracketEll}. The Serre spectral sequence for this extension gives an exact sequence
\[\cdots\lra \left[\begin{cases}
0 & \text{ $n$ odd}\\
\mr{Hom}(H_n, \bZ/2) &\text{ $n$ even}
\end{cases}\right]_{G_g^{\fr,[[\ell]]}} \lra H_1({\Gamma}_g^{\fr, [[\ell]]}/\Theta_{2n+1}) \lra H_1(G_g^{\fr,[[\ell]]}) \lra 0\]
and it follows from \cite[Lemma A.2]{KrannichMCG} that the leftmost term is zero. Thus we have $H_1({\Gamma}_g^{\fr, [[\ell]]}/\Theta_{2n+1}) \overset{\sim}\to H_1(G_g^{\fr,[[\ell]]})$ from which the factorisation of $g_\ell$ over some $h_\ell$ follows.

Before proving (ii) and (iii), we first show that for large enough $g$ the group ${\Gamma}_g^{\fr, [[\ell]]}/{\Gamma}_g^{\fr, [\ell]}$ has order at least 4. To see this we use the quotient map to ${G}_g^{\fr, [[\ell]]}/{G}_g^{\fr, [\ell]}$. By Lemma \ref{lem:triv-fr-stable}, for $n \neq 1, 3$ the group ${G}_g^{\fr, [\ell]}$ has trivial abelianisation, so there is an induced surjection
\[{G}_g^{\fr, [[\ell]]}/{G}_g^{\fr, [\ell]} \lra H_1({G}_g^{\fr, [[\ell]]}) = \begin{cases}
\bZ/4 & \text{$n$ odd,}\\
(\bZ/2)^2 & \text{$n$ even}.
\end{cases}\]
Thus  ${\Gamma}_g^{\fr, [[\ell]]}/{\Gamma}_g^{\fr, [\ell]}$ indeed has order at least 4. By the exact sequence \eqref{eq:HatF} we have ${\Gamma}_g^{\fr, [[\ell]]}/{\Gamma}_g^{\fr, [\ell]} = \mr{im}({f}_\ell) = \mr{im}(h_\ell)$, so this group has order at least 4.

On the other hand $ H_1({G}_\infty^{\fr, [[\ell]]})$ has order precisely 4, so $h_\ell$ must be injective. If $n \not\equiv 0 \pmod 4$ and $n \neq 3$ then $\pi_{2n}(\mr{SO}(2n))$ has order 4, so $h_\ell$ must be an isomorphism.  If $n \equiv 0 \pmod 4$ then $\pi_{2n}(\mr{SO}(2n))$ has order 8, so $h_\ell$ must be injective onto an index 2 subgroup. This proves parts (ii) and (iii).
\end{proof}

This discussion allows us to describe the subgroups $G_g^{\fr, [\ell]} \leq G_g^{\fr, [[\ell]]}$, where they have the following interesting intrinsic descriptions.

\begin{corollary}\label{cor:IdentifyingSingleBracketEll}
If $n \neq 1, 3$ then $G_g^{\fr, [\ell]} = \mr{ker}(G_g^{\fr, [[\ell]]} \to H_1(G_\infty^{\fr, [[\ell]]}))$.

If $n=1, 3$ then $G_g^{\fr, [\ell]} = G_g^{\fr, [[\ell]]}$.
\end{corollary}
\begin{proof}
For $n \neq 1,3$ the kernel of the composition
\[\Gamma_g^{\fr, [[\ell]]} \lra G_g^{\fr, [[\ell]]} \lra H_1(G_\infty^{\fr, [[\ell]]})\]
is the same as the kernel $\Gamma_g^{\fr, [\ell]}$ of ${f}_\ell$ by Proposition \ref{prop:FEllAnalysis}. Thus $G_g^{\fr, [[\ell]]} \to H_1(G_\infty^{\fr, [[\ell]]})$ has kernel $G_g^{\fr, [\ell]}$. 

If $n=1, 3$ then $\pi_{2n}(\mr{SO}(2n))=0$ so $\Gamma_g^{\fr, [\ell]}=\Gamma_g^{\fr, [[\ell]]}$ and hence their images in $G'_g$ are equal too.
\end{proof}

\begin{proof}[Proof of Theorem \ref{thm:main}]
Suppose that $n \geq 2$ and $g \geq 1$. The argument will be by analysing the sequence \eqref{eq:FundSeq} using the factorisation of Proposition \ref{prop:FEllAnalysis}. This yields the exact top row of the commutative diagram 
\begin{equation*}
\begin{tikzcd} 
0 \rar &\Gamma^{\mr{fr},[\ell]}_g \rar &[-18pt] \Gamma^{\mr{fr},[[\ell]]}_g \rar{f_\ell} \dar[two heads] &[-11pt] \mr{Str}^\mr{fr}_\partial(D^{2n}) \rar \dar{\cong} &[-15pt] \mr{Str}^\mr{fr}_\partial(W_{g,1})/\Gamma_g \rar[two heads] &[-11pt] \mr{Str}_\ast^\mr{fr}(W_{g,1})/\Gamma_g \\
&	 & H_1(G^{\mr{fr},[[\ell]]}_\infty) \rar{h_\ell} & \pi_{2n}(\mr{SO}(2n)).
	\end{tikzcd}
	\end{equation*}
	
 By the calculations in Section \ref{sec:Abs} the composition
\[{\Gamma}_g^{\fr, [[\ell]]} \lra {G}_g^{\fr, [[\ell]]} \lra H_1({G}_g^{\fr, [[\ell]]}) \lra H_1({G}_\infty^{\fr, [[\ell]]}),\]
which is the left-hand vertical map in the diagram, is surjective as long as $g \geq 1$. By Proposition \ref{prop:FEllAnalysis} if $n \not \equiv 0 \pmod 4$ then the map ${f}_\ell$ is surjective; if $n \equiv 0 \pmod 4$ then it has cokernel $\bZ/2$. By Proposition \ref{prop:upper-bound} the set $\mr{Str}^\mr{fr}_\ast(W_{g,1})/\Gamma_g$, which is the right-most term in the diagram, has a single element unless $n=3,7$ in which case it has two. It follows by combining these facts that $\mr{Str}^\mr{fr}_\partial(W_{g,1})/\Gamma_g$ has a single element unless $n = 3, 7$ or $n \equiv 0 \pmod 4$ in which case it has two elements.

The case $n=1$ and $g \geq 2$ is \cite[Theorem 2.9]{RWFramed}.
\end{proof}

\section{$\theta$-structures on $W_{g,1}$}\label{sec:ThetaStr}

In \cite[Section 8]{KR-WAlg} we more generally considered tangential structures whose $\mr{GL}_{2n}(\bR)$-space $\Theta$ has the property that the homotopy quotient $B \coloneqq \Theta \sslash \mr{GL}_{2n}(\bR)$ is $n$-connected (in terms of the associated fibration $\theta \colon B \to B\mr{O}(2n)$, this means that $B$ is $n$-connected). Our results about framings can be used to also classify such $\theta$-structures on $W_{g,1}$, up to homotopy and diffeomorphisms. We will assume that $n \geq 2$.

Given a boundary condition $\ell_\partial\colon \mr{Fr} (TW_{g,1}\vert_{\partial W_{g,1}}) \to \Theta$  we let $\mr{Bun}_\partial(\mr{Fr}(TW_{g,1}),\Theta;\ell_\partial)$ be the space of $\mr{GL}_{2n}(\bR)$-maps $\mr{Fr}(TW_{g,1}) \to \Theta$ extending $\ell_\partial$. Its set of path components is denoted $\mr{Str}^\theta_\partial(W_{g,1})$. The moduli space of $W_{g,1}$'s with $\theta$-structures is defined as
\[B\mr{Diff}^\theta_\partial(W_{g,1}) \coloneqq \mr{Bun}_\partial(\mr{Fr}(TW_{g,1}),\Theta;\ell_\partial) \sslash \mr{Diff}_\partial(W_{g,1}).\]
Its set of path components is the set of orbits $\mr{Str}^\theta_\partial(W_{g,1})/\Gamma_g$. This is what we shall compute in this section, but we first make some definitions.

The boundary condition $\ell_\partial$ singles out a path component $\Theta^+$ of $\Theta$. Since $\Theta \sslash \mr{GL}_{2n}(\bR)$ is $n$-connected, the map $\pi_n(\mr{SO}(2n)) \to \pi_n(\Theta^+)$ is surjective. By Lemma \ref{lem:levine} the map $S\pi_n(\mr{SO}(n)) \to \pi_n(\mr{SO}(2n))$ is surjective unless $n=3,7$, in which case it has cokernel $\bZ/2$. This leads to two cases when $n=3,7$:
\begin{enumerate}[(A)]
	\item $S\pi_n(\mr{SO}(n)) \to \pi_n(\mr{SO}(2n)) \to \pi_n(\Theta^+)$ is not surjective (and thus has index 2),
	\item $S\pi_n(\mr{SO}(n)) \to \pi_n(\mr{SO}(2n)) \to \pi_n(\Theta^+)$ is surjective.
\end{enumerate}

We also define
\[C\pi_{2n}(\Theta^+) \coloneqq \mr{coker}\Big(H_1(G^{\mr{fr},[[\tau]]}_\infty) \xrightarrow{h_\tau} \pi_{2n}(\mr{SO}(2n)) \to \pi_{2n}(\Theta^+)\Big),\]
which seems at first sight to depend on the orbit of homotopy class $[[\tau]]$ of reference framing $\tau$ relative to $\ast$, but is in fact independent of this choice: by Proposition \ref{prop:FEllAnalysis} the map $h_\tau$ is surjective unless $n \equiv 0 \pmod 4$, but in that case the orbit of $[[\tau]]$ is unique by Proposition \ref{prop:upper-bound}.

\begin{theorem}\label{thm:theta}
Suppose that $\Theta \sslash \mr{GL}_{2n}(\bR)$ is $n$-connected. Let $g \geq 1$ and $n \geq 2$. The set of orbits for the action of the mapping class group $\Gamma_g$ on the set $\mathrm{Str}_\partial^{\theta}(W_{g,1})$ of homotopy classes of $\theta$-structures extending $\ell_\partial$ is in bijection with
	\begin{enumerate}[(i)]
		\item $C\pi_{2n}(\Theta^+)$ if $n \neq 3,7$;
		\item $C\pi_{2n}(\Theta^+) \times \bZ/2$ if $n=3,7$ and we are in case (A);
		\item $\pi_{2n}(\Theta^+)$ if $n=3,7$ and we are in case (B).
	\end{enumerate}
\end{theorem}

We will explain the proof of this theorem in parallel with that of Theorem \ref{thm:main}. The definitions and results of Sections \ref{sec:framings}, \ref{sec:moduli-spaces}, and \ref{sec:Relaxing} go through for $\theta$-structures. By \cite[Lemma 8.5]{KR-WAlg}, up to homotopy there is a unique orientation preserving boundary condition $\ell_\partial$ which extends to a $\theta$-structure on all of $W_{g,1}$, which we may take to be $\ell^\tau_\partial$ coming from a reference framing $\tau$. The reference framing induces a homeomorphism
\[\mr{Bun}_\partial(\mr{Fr}(TW_{g,1}),\Theta;\ell_\partial) \cong \mr{map}_\partial(W_{g,1},\Theta).\]
Its path components will be denoted $\mr{Str}^\theta_\partial(W_{g,1})$. We can relax boundary conditions to get an exact sequence
\[\begin{tikzcd} 
0 \rar &[-12pt] \mr{Str}^\theta_\partial(D^{2n}) \dar{\cong} \rar{\circlearrowright} &[-9pt] \mr{Str}^\theta_\partial(W_{g,1}) \rar \dar{\cong} &[-9pt] \mr{Str}^\theta_\ast(W_{g,1}) \rar \dar{\cong} &[-12pt] 0 \\
0 \rar & \pi_{2n}(\Theta^+) \rar{\circlearrowright} &  \pi_0(\mr{map}_\partial(W_{g,1}, \Theta)) \rar & \mr{Hom}(H_n,\pi_n(\Theta^+)) \rar & 0.
\end{tikzcd}\]

The arguments for Proposition \ref{prop:upper-bound} go through for $\theta$-structures, giving the following:

\begin{proposition}\label{prop:upper-bound-theta} Suppose $n \geq 2$ and $g \geq 1$, then
	\begin{itemize}
		\item If $n \neq 3,7$ or we are in case (B), then $\mr{Str}^\theta_\ast(W_{g,1})/\Gamma_g$ consists of a single element.
		\item If $n=3,7$ and we are in case (A), then $\mr{Str}^\theta_\ast(W_{g,1})/\Gamma_g$ consists of two elements.
	\end{itemize}
\end{proposition}

These arguments also give information about the stabiliser $\Gamma_g^{\theta,[[\ell]]}$ of $[[\ell]] \in \mr{Str}^\theta_\ast(W_{g,1})$, as well as its image $G^{\theta,[[\ell]]}_g$ in $G'_g$. The map
\begin{align*}
\mr{Str}^{\fr}_\ast(W_{g,1}) &\lra \mr{Str}^\theta_\ast(W_{g,1})\\
[[\ell^{\fr}]] &\longmapsto [[\ell]]
\end{align*}
associating to a framing the induced $\theta$-structure is surjective, as by assumption the map $\pi_n(\mr{SO}(2n)) \to \pi_n(\Theta^+)$ is. If $n=2$ it follows that $\mr{Str}^\theta_\ast(W_{g,1})$ is a single point; if $n \neq 3, 7$ it follows that $\mr{Str}^\theta_\ast(W_{g,1})/I_g$ is a single point; if $n=3,7$, it follows as in the proof of Proposition \ref{prop:framings-rel-point} that $\mr{Str}^\theta_\ast(W_{g,1})/I_g$ is in bijection with $\mr{Quad}(H_n, \lambda)$ in case (A) and is a single point in case (B). Thus for $n \geq 2$ and $g \geq 1$, we have
	\[G_g^{\theta, [[\ell]]} = \begin{cases} 
	\mr{Sp}_{2g}^{q\text{ or }a}(\bZ) & \text{if $n=3,7$, we are in case (A), and $\ell$ has Arf invariant $0$ or $1$,} \\
	\mr{Sp}_{2g}(\bZ) & \text{if $n = 3,7$ and we are in case (B),} \\
	\mr{Sp}_{2g}^q(\bZ) & \text{if $n$ is odd but not 3 or 7,} \\
	\mr{O}_{g,g}(\bZ) & \text{if $n$ is even.}\end{cases}\]
	Furthermore, if $n \neq 3,7$ or $n=3,7$ and we are in case (A), it follows that the natural inclusion of stabilisers $\Gamma_g^{\mr{fr},[[\ell^\mr{fr}]]} \to \Gamma_g^{\theta,[[\ell]]}$	is an isomorphism.

The analogue of the fundamental sequence \eqref{eq:FundSeq} for $\theta$-structures gives, for any framing $\ell^{\fr}$, a commutative diagram
\[\begin{tikzcd} 
 \Gamma^{\fr,[\ell^\fr]}_g \arrow[r, hook] \dar & \Gamma^{\fr,[[\ell^\fr]]}_g \rar{{f}_{\ell^\fr}} \dar{(\ast)} & \mr{Str}^\fr_{\partial}(D^{2n}) \dar \rar{- \cdot [\ell^{\fr}]} &\mr{Str}^\fr_\partial(W_{g,1})/\Gamma_g \arrow[r, two heads] \dar & \mr{Str}^\fr_{*}(W_{g,1})/\Gamma_g \dar\\
 \Gamma^{\theta,[\ell]}_g \arrow[r, hook] & \Gamma^{\theta,[[\ell]]}_g \rar{{f}_{\ell}} & \mr{Str}^\theta_{\partial}(D^{2n}) \rar{- \cdot [\ell]} &\mr{Str}^\theta_\partial(W_{g,1})/\Gamma_g \arrow[r, two heads] & \mr{Str}^\theta_{*}(W_{g,1})/\Gamma_g.
\end{tikzcd}\]

\begin{proof}[Proof of Theorem \ref{thm:theta}]
We proceed in three cases.
\vspace{1ex}

\noindent\textbf{The cases $n \neq 3,7$.} In this case the map indicated by ($\ast$) is an isomorphism and ${f}_\ell$ is determined by ${f}_{\ell^\fr}$. Using Proposition \ref{prop:FEllAnalysis}, we then identify the cokernel of ${f}_\ell$ with $C\pi_{2n}(\Theta^+)$. As $\mr{Str}^\theta_\ast(W_{g,1})/\Gamma_g$ consists of a single element, we conclude that $\mr{Str}^\theta_\partial(W_{g,1})/\Gamma_g \cong C\pi_{2n}(\Theta^+)$.
\vspace{1ex}

\noindent\textbf{Case (A).} In this case $n=3$ or $7$ and $S\pi_n(\mr{SO}(n)) \to \pi_n(\mr{SO}(2n)) \to \pi_n(\Theta)$ not surjective. Then $\mr{Str}^\theta_\ast(W_{g,1})/\Gamma_g$ consists of two elements, distinguished by an Arf invariant. Choosing framings $\ell^\mr{fr}$ with Arf invariant $0$ and $1$ respectively, we get two commutative diagrams as above. In both cases, the map indicated by ($\ast$) is an isomorphism, and as above we identify the cokernel of ${f}_\ell$ with $C\pi_{2n}(\Theta^+)$. Thus we get a collection of orbits with Arf invariant $0$ and another collection of orbits with Arf invariant $1$, each in bijection with $C\pi_{2n}(\Theta^+)$.
\vspace{1ex}

\noindent\textbf{Case (B).} In this case $n=3$ or $7$ and $S\pi_n(\mr{SO}(n)) \to \pi_n(\mr{SO}(2n)) \to \pi_n(\Theta)$ surjective. The proof of Proposition \ref{prop:FEllAnalysis} gives a factorisation
\[\begin{tikzcd} \Gamma^{\theta,[[\ell]]}_g \arrow{rr}{{f}_\ell} \dar & & \pi_{2n}(\Theta) \\
G^{\theta,[[\ell]]}_g \rar & H_1(G^{\theta,[[\ell]]}_g) \rar & H_1(G^{\theta,[[\ell]]}_\infty) \uar{h_\ell} \end{tikzcd}\]
with left map the quotient map, and bottom maps abelianisation followed by stabilisation. As $G^{\theta,[[\ell]]}_g = \mr{Sp}_{2g}(\bZ)$, the bottom-right term vanishes by the computations in Section \ref{sec:Abs}, so ${f}_\ell=0$. Combining this with the fact that $\mr{Str}^\theta_{*}(W_{g,1})/\Gamma_g$ is a single point in this case gives the claimed result.
\end{proof}

\subsection{Example: stable framings}

\emph{Stable framings} are trivialisations of the stable tangent bundle. In this case we take $\Theta = \mr{GL}_\infty(\bR)$, made into a $\mr{GL}_{2n}(\bR)$-space by stabilisation. Then $\Theta^+ = \mr{GL}^+_\infty(\bR)$, which deformation retracts onto $\mr{SO}$. The map $\pi_n(\mr{SO}(2n)) \to \pi_n(\mr{SO})$ induced by stabilisation is an isomorphism as long as $n \geq 2$, so when $n=3,7$ we are in case (A). 

\begin{lemma}
The map $f_\ell\colon \Gamma^{\mr{sfr}, [[\ell]]}_g \to \pi_{2n}(\mr{SO})$ is zero.
\end{lemma}
\begin{proof}
The group $\pi_{2n}(\mr{SO})$ vanishes unless $n \equiv 0 \pmod 4$ in which case it is given by $\bZ/2$. In this case we claim that $h_\ell\colon H_1(G^{\mr{sfr},[[\ell]]}_\infty;\bZ) \to  \pi_{2n}(\mr{SO})$ is zero. To see this note that the composition
\[H_1(\check{\Gamma}^{\mr{sfr},\ell}_\infty;\bZ) \lra H_1(G^{\mr{sfr},[\ell]}_\infty;\bZ) \lra  H_1(G^{\mr{sfr},[[\ell]]}_\infty;\bZ) \overset{h_\ell}\lra  \pi_{2n}(\mr{SO})\]
is zero by the analogue for stable framings of the exact sequence \eqref{eq:HatF}. But $G^{\mr{sfr},[\ell]}_g = G^{\mr{sfr},[[\ell]]}_g = \mr{O}_{g,g}(\bZ)$ for $n$ even as we have discussed above, and by Section 5.2 of \cite{grwabelian} the composition
		\[\pi_{1}(\Sigma^{-2n}\mr{SO}/\mr{SO}(2n)) \cong H_1(\check{\Gamma}^{\mr{sfr},\ell}_\infty;\bZ) \lra \pi_1(\mathbf{MT}\theta_n) \cong H_1(\Gamma_\infty;\bZ) \lra H_1(\mr{O}_{\infty,\infty}(\bZ);\bZ) \]
is surjective.
\end{proof}

We conclude that there are, up to homotopy and diffeomorphism, two stable framings on $W_{g,1}$  when $n=3,7$ or $n \equiv 0 \pmod 4$, and a unique one otherwise. That is, the classification of stable framings is the same as that of framings: up to homotopy and diffeomorphism, every stable framing of $W_{g,1}$ arises from a unique framing.

\begin{remark}\label{rem:index-4-ident} 
Returning to Proposition \ref{prop:FEllAnalysis} (iii), i.e.~the case $n \equiv 0 \pmod 4$, this argument identifies the index 2 subgroup hit by the map called $h_\ell$ in that proposition, which in the notation of this section would be called $h_{\ell^{\fr}}$. It implies that this index 2 subgroup is the kernel of the stabilisation map $\pi_{2n}(\mr{SO}(2n)) \to \pi_{2n}(\mr{SO})$. To see this consider the commutative diagram
\[\begin{tikzcd}H_1(G^{\mr{fr},[[\ell]]}_\infty) \cong (\bZ/2)^2 \rar[hook]{h_{\ell^{\fr}}} \dar & \pi_{2n}(\mr{SO}(2n)) \cong (\bZ/2)^3 \dar[two heads] \\
	H_1(G^{\mr{sfr},[[\ell]]}_\infty) \rar{h_\ell = 0} & \pi_{2n}(\mr{SO}) \cong \bZ/2 \end{tikzcd}\]
	where the right-hand vertical map is surjective by Lemma \ref{lem:levine} (i).
\end{remark}

\bibliographystyle{amsplain}
\bibliography{../../cell}

\vspace{.5cm}

\end{document}